\documentclass[a4paper,final,reqno]{amsart}

\usepackage{bm,enumerate,nicefrac}
\usepackage[foot]{amsaddr}
\usepackage{amsthm,amssymb,amsfonts}

\usepackage{fourier}
\usepackage[T1]{fontenc}

\newtheorem{theorem}{Theorem}[section]
\newtheorem{proposition}[theorem]{Proposition}

\newtheorem{corollary}[theorem]{Corollary}
\newtheorem{definition}[theorem]{Definition}
\theoremstyle{definition}
\newtheorem{remark}[theorem]{Remark}
\newtheorem{example}[theorem]{Example}
\newtheorem{examples}[theorem]{Examples}

\renewcommand{\ge}{\geqslant}
\renewcommand{\le}{\leqslant}

\newcommand{\set}[1]{\mathcal{#1}}
\newcommand{\mat}[1]{\bm{\mathsf{#1}}}
\newcommand{\imat}[1]{\mat{#1}^{-1}}

\newcommand{\I}{\mat I_m}

\newcommand{\abs}[1]{\left|#1\right|}

\newcommand{\norm}[1]{\left\|#1\right\|_\infty}
\newcommand{\sr}[1]{\varrho\left(#1\right)}
\newcommand{\C}{\mathbb{C}^{m}}
\newcommand{\R}{\mathbb{R}^{m}}

\newcommand{\CC}{\mathbb{C}^{m \times m}}
\newcommand{\RR}{\mathbb{R}^{m \times m}}

\DeclareMathOperator{\diag}{diag}
\DeclareMathOperator{\off}{off}
\DeclareMathOperator{\tril}{tril}
\DeclareMathOperator{\triu}{triu}

\newcommand{\matle}{\preceq}
\newcommand{\matge}{\succeq}
\renewcommand{\t}{\vartheta}
\newcommand{\emptymat}{[\raisebox{0.5pt}{$\,\emptymatn\,$}]}
\newcommand{\emptymatn}{\star}

\title[The Sassenfeld criterion and H-matrices]{A note on the Sassenfeld criterion and its relation to H-matrices}

\author[T.~P.~Wihler]{Thomas P. Wihler}
\address{Mathematics Institute, University of Bern, Switzerland}


\makeatletter
\@namedef{subjclassname@2020}{%
  \textup{2020} Mathematics Subject Classification}
\makeatother


\begin{document}

\begin{abstract}
The starting point of this note is a decades-old yet little-noticed sufficient condition, presented by Sassenfeld in 1951, for the convergence of the classical Gau\ss-Seidel method. The purpose of the present paper is to shed new light on \emph{Sassenfeld's criterion} and to demonstrate that it is directly related to \emph{H-matrices}. In particular, our results yield a new characterization of H-matrices. In addition, the convergence of iterative linear solvers that involve H-matrix preconditioners is briefly discussed.
\end{abstract}

\maketitle

\section{Introduction}

The Gau\ss-Seidel method is amongst the most classical iterative schemes for the solution of systems of linear equations.  Traditionally, in many numerical analysis textbooks, convergence is established for matrices that are either strictly diagonally dominant or symmetric positive definite. Only a few authors (see, e.g., \cite[Thm.~4.16]{Wendland:17}) point to a less standard convergence criterion for the Gau\ss-Seidel scheme that was introduced by Sassenfeld in his paper~\cite{Sassenfeld:51}: Given a matrix $\mat A=[a_{ij}]\in\CC$ with non-vanishing diagonal entries, i.e.~ $a_{ii}\neq 0$ for each $i=1,\ldots,m$,  define non-negative real numbers $s_1,\ldots,s_m$ iteratively by
\begin{equation}\label{eq:SF}
s_i=\frac{1}{|a_{ii}|}\Bigg(\sum_{j<i}|a_{ij}|s_j+\sum_{j>i}|a_{ij}|\Bigg),\qquad i=1,\ldots,m.
\end{equation}
Sassenfeld has proved that the condition
\begin{equation}\label{eq:s01}
0\le s_i<1\qquad \forall~i=1,\ldots,m,
\end{equation}
is sufficient for the convergence of the Gau\ss-Seidel iteration. Matrices that satisfy this property (which is closely related to generalized diagonal dominance, see, e.g., \cite{JamesRiha:74}) were discussed recently in~\cite{BaumannWihler:17}.

The purpose of the present note is to show that there is a more general principle behind Sassenfeld's original work that is intimately related to H-matrices. To illustrate this observation, we note that~\eqref{eq:SF} can be written in matrix form as
\begin{equation}\label{eq:SF'}
(\abs{\mat D}-\abs{\mat L})\mat s=\abs{\mat U}\mat e,
\end{equation}
where the matrix $\mat A=\mat L+\mat D+\mat U$ is decomposed in the usual way into the (strict) lower and upper triangular parts $\mat L=\tril(\mat A)$ and $\mat U=\triu(\mat A)$, respectively, and the diagonal part $\mat D=\diag(\mat A)$; furthermore, $\abs{\emptymat}$ signifies the modulus of a matrix $\emptymat$ taken entry-wise, $\mat s=(s_1,\ldots,s_m)$ is a vector that contains the iteratively defined real numbers $s_1,\ldots,s_m$ from~\eqref{eq:SF}, and 
\begin{equation}\label{eq:e}
\mat e=(1,\ldots,1)^\top\in\R
\end{equation} 
is the (column) vector with all components~1. More generally, for appropriate matrices $\mat P\in\CC$, we consider the splitting 
\[
\mat A=\off(\mat P)+\diag(\mat P)+(\mat A-\mat P),
\]
where $\off(\emptymat)$ denotes the off-diagonal part of a matrix. Then, define the vector $\mat s\in\R$ to be the solution (if it exists) of the system
\begin{equation}\label{eq:SF''}
\left(\abs{\diag(\mat P)}-\abs{\off(\mat P)}\right)\mat s=\abs{\mat A-\mat P}\mat e.
\end{equation} 
For instance, in the context of the Gau\ss-Seidel scheme, letting $\mat P:=\mat L+\mat D$, with $\mat L$ and $\mat D$ as above, we notice that~\eqref{eq:SF''} translates immediately into~\eqref{eq:SF'}. In this work, we will focus on matrices $\mat A$ and $\mat P$ for which the components of the solution vector $\mat s$ of the linear system~\eqref{eq:SF''} satisfy the Sassenfeld criterion~\eqref{eq:s01}.

\subsection*{Outline}

We begin our work by reviewing the class of H-matrices, see~\S\ref{sc:Hmatrices}, which was originally introduced in \cite{Ostrowski:37}, and plays a crucial role in the convergence of iterative splitting methods (especially, the Jacobi, Gau\ss-Seidel, and SOR schemes); in the context of this paper, such matrices are exactly those for which the system~\eqref{eq:SF''} is non-singular. In \S\ref{sc:SI} we continue by introducing the so-called \emph{Sassenfeld index}, which is an essential quantity for our analysis, and derive some basic estimates. Subsequently, in \S\ref{sc:sm}, based on the previously defined Sassenfeld index, we will focus on all matrices for which the bounds~\eqref{eq:s01} for the solution vector $\mat s$ of~\eqref{eq:SF''} can be achieved; such matrices will be termed \emph{Sassenfeld matrices}. Our main result (Thm.~\ref{thm:main}) will establish an equivalence for  Sassenfeld matrices and  (non-singular) H-matrices; in this regard, our work provides a new characterization of H-matrices. In addition, a computational verification procedure is proposed (see Prop.~\ref{prop:it}); cf.~the related papers~\cite{OjiroNikiUsui:03,BruGimenezHadjidimos:12}. Finally, we conclude this article with a few remarks in~\S\ref{sc:concl}.

\subsection*{Notation}
For any vectors or matrices $\mat X,\mat Y\in\mathbb{R}^{m\times n}$, we use the notation $\mat X\matge\mat Y$ (or $\mat X\succ\mat Y$) to indicate that all entries of the difference $\mat X-\mat Y\in\mathbb{R}^{m\times n}$ are non-negative (resp.~positive). Furthermore, for a matrix $\mat A=[a_{ij}]\in\mathbb{R}^{m\times n}$, we denote by 
$
\norm{\mat A}:=\max_{1\le i\le m}\sum_{j=1}^n|a_{ij}| 
$
the standard $\infty$-norm. Moreover, we signify by $\sr{\mat A}$ the spectral radius of a square-matrix $\mat A\in\CC$, and $\I\in\CC$ is the identity matrix.

\section{A brief review of H-matrices}\label{sc:Hmatrices}

We will denote by $\set{H}_m$ the subset of all H-matrices in $\CC$. This set was originally introduced in~\cite{Ostrowski:37}, see also~\cite{BCGM:08,BP:94,Varga:00}, and consists of all matrices $\mat A=[a_{ij}]\in\CC$ for which the associated comparison matrix, given by
\[
\mathfrak{M}(\mat A):=
\abs{\diag(\mat A)}-\abs{\off(\mat A)}
=\begin{cases}
-|a_{ij}|,&\text{if }i\neq j,\\
+|a_{ii}|,&\text{if }i=j,
\end{cases}\qquad 1\le i,j\le m,
\]
is a non-singular M-matrix, i.e. it takes the form $\mathfrak{M}(\mat A)=r\I-\mat B$, for a matrix $\mat B\matge\mat 0$, with $r>\sr{\mat B}$. \medskip

We collect a few well-known facts about H-matrices that are instrumental for the present work.
\begin{enumerate}[(F1)]

\item
We first remark that matrices in $\set{H}_m$ are non-singular. Indeed, suppose to the contrary that there is $\mat A\in\set{H}_m$ and a vector $\mat x\in\C$ with $\|\mat x\|_\infty=1$ and $\mat A\mat x=\mat 0$. Then, it follows that 
\[
\mathfrak{M}(\mat A)\abs{\mat x}
=\abs{\diag(\mat A)\mat x}-\abs{\off(\mat A)}\abs{\mat x}
=\abs{-\off(\mat A)\mat x}-\abs{\off(\mat A)}\abs{\mat x}\matle\mat 0,
\]
and thus $r\abs{\mat x}
=\mathfrak{M}(\mat A)\abs{\mat x}+\mat B\abs{\mat x}
\matle\mat B\abs{\mat x}$, with $\mat B\matge\mat 0$ as above. This
implies that $\abs{\mat x}\matle r^{-1}\mat B\abs{\mat x}$. Iteratively, for any $n\in\mathbb{N}$, we infer that $\abs{\mat x}\matle \left(r^{-1}\mat B\right)^n\abs{\mat x}$. Exploiting that $\sr{r^{-1}\mat B}<1$ and letting $n\to\infty$, we deduce that $\mat x=\mat 0$, a contradiction. \medskip

\item It is well-known, see e.g.~\cite[Thm.~5']{Fan:58}, that $\mat A\in\set{H}_m$ if and only if there is a positive real vector $\mat u\succ\mat 0$ such that $\mathfrak{M}(\mat A)\mat u\succ\mat 0$; in individual components, this means that there are positive numbers $u_1,\ldots,u_m>0$ such that
\[
|a_{ii}|u_i>\sum_{j\neq i}|a_{ij}|u_j\qquad\forall i=1,\ldots,m;
\]
incidentally, this property refers to the notion of generalized diagonal dominance (by rows); cf., e.g., ~\cite{JamesRiha:74}. In particular, the above bound implies that the diagonal entries of any matrix in $\set{H}_m$ are all non-zero.\medskip

\item
Furthermore, for $\mat A\in\set{H}_m$, the inverse matrix of $\mathfrak{M}(\mat A)$ exists and is non-negative; indeed, since $\sr{r^{-1}\mat B}<1$, with $\mat B$ from above, we have
\[
\mathfrak{M}(\mat A)^{-1}
=\left(r\I-\mat B\right)^{-1}
=r^{-1}\left(\I-r^{-1}\mat B\right)^{-1}
=\sum_{k\ge 0}r^{-1-k}\mat B^k\matge\mat 0.
\]

\item Finally, for any matrix $\mat A\in\set{H}_m$, it holds that
\[
\sr{\abs{\diag(\mat A)}^{-1}\abs{\off(\mat A)}}
=
\sr{\diag(\mathfrak{M}(\mat A))^{-1}\off(\mathfrak{M}(\mat A)}<1;
\]
see, e.g.~\cite[Thm.~1~(vii)]{Varga:76}. 

\end{enumerate}

\section{Sassenfeld index}\label{sc:SI}

For a non-singular matrix $\mat A\in\CC$, a right-hand side vector $\mat b\in\C$, and an arbitrary starting vector $\mat x_0\in\C$, we will be interested in the iterative splitting scheme
\begin{equation}\label{eq:it}
\mat P\mat x_{n+1}=(\mat P-\mat A)\mat x_n+\mat b,\qquad n\ge 0,
\end{equation}
for the solution of the linear system
\begin{equation}\label{eq:Ax=b}
\mat A\mat x=\mat b.
\end{equation}
The focus of this work will be on preconditioners $\mat P\in\set{H}_m$. \medskip

From fact  (F3) above, for $\mat P\in\mathcal{H}_m$, we infer that the vector
defined by
\begin{equation}\label{eq:s}
\mat s(\mat A,\mat P):=\mathfrak{M}(\mat P)^{-1}\abs{\mat A-\mat P}\mat e\matge\mat 0,
\end{equation}
with~$\mat e\in\mathbb{R}^m$ from~\eqref{eq:e}, is well-defined and contains only non-negative components. 

\begin{definition}[Sassenfeld index]
The \emph{Sassenfeld index} of a matrix $\mat A\in\CC$ with respect to a preconditioner $\mat P\in\set{H}_m$ is defined by $\mu(\mat A,\mat P):=\norm{\mat s(\mat A,\mat P)}$, with the vector $\mat s(\mat A,\mat P)$ from~\eqref{eq:s}.
\end{definition}

The essence of the Sassenfeld index defined above is that it allows to control the norm $\norm{\I-\imat P\mat A}$ of the iteration matrix in the splitting method~\eqref{eq:it} in a non-standard way.
 
\begin{proposition}\label{pr:bound}
Let $\mat A\in\CC$ be a non-singular matrix, and $\mat P\in\set{H}_m$. Then, it holds that
\begin{equation}\label{eq:bound}
\norm{\I-\imat P\mat A}
\le\mu(\mat A,\mat P).
\end{equation}
\end{proposition}

\begin{proof}
Consider an arbitrary vector $\mat y\in\C$ with $\norm{\mat y}=1$. Defining $\mat R=\mat P-\mat A$, we let
\begin{equation}\label{eq:aux20201006a}
\mat x:=\imat P\mat R\mat y=\imat P(\mat P-\mat A)\mat y=(\I-\imat P\mat A)\mat y.
\end{equation}
Note first that
$
\diag(\mat P)\mat x+\off(\mat P)\mat x=\mat R\mat y.
$
Taking moduli results in
\[
\mathfrak{M}(\mat P)\abs{\mat x}\matle\abs{\mat R}\abs{\mat y}\matle\abs{\mat R}\mat e.
\]
Recalling that $\mathfrak{M}(\mat P)^{-1}\matge\mat 0$, cf.~fact (F3) above, and employing~\eqref{eq:s}, we deduce that
\[
\abs{\mat x}\matle\mathfrak{M}(\mat P)^{-1}\abs{\mat R}\mat e=\mat s(\mat A,\mat P).
\]
Therefore, using~\eqref{eq:aux20201006a}, we infer that
\[
\norm{(\I-\imat P\mat A)\mat y}=\norm{\mat x}\le\norm{\mat s(\mat A,\mat P)},
\]
which yields~\eqref{eq:bound}.
\end{proof}

\begin{corollary}[Invertibility]\label{cor:inv}
Given a matrix $\mat A$, and a preconditioner $\mat P\in\set{H}_m$. Then, the matrix $\mat A_\tau=\mat A+\tau\mat P$ is non-singular whenever~$|\tau+1|>\mu(\mat A,\mat P)$. 
\end{corollary}

\begin{proof}
We apply a contradiction argument. To this end, suppose that there exists~$\mat v\in\C$, $\norm{\mat v}=1$, such that~$\mat A_\tau\mat v=\mat 0$. Then, it holds that $(\tau+1)\mat P\mat v=(\mat P-\mat A)\mat v$, and thus $(\tau+1)\,\mat v=\mat P^{-1}(\mat P-\mat A)\mat v$. Taking norms, and using~\eqref{eq:bound}, yields
\[
|\tau+1|=
\norm{(\I-\imat P\mat A)\mat v}
\le\norm{\I-\imat P\mat A}
\le\mu(\mat A,\mat P),
\]
which causes a contradiction to the range of $\tau$.
\end{proof}
  
We note that the vector $\mat s(\mat A,\mat P)$ from~\eqref{eq:s} can be computed approximately by iteration. Indeed, if $\mat P\in\set{H}_m$, then the diagonal entries of $\mat P$ do not vanish, cf.~fact (F2) above, and the iterative scheme given by
\begin{equation}\label{eq:sk}
\abs{\diag(\mat P)}\mat s^{(k+1)}=\abs{\off(\mat P)}\mat s^{(k)}+\abs{\mat A-\mat P}\mat e,\qquad k\ge 0,
\end{equation}
converges to the vector $\mat s(\mat A,\mat P)$ from~\eqref{eq:s} for any initial vector $\mat s^{(0)}\in\C$. Furthermore, the following result provides a computational upper bound for the Sassenfeld index.

\begin{proposition}[Iterative estimation of the Sassenfeld index]\label{prop:it}
Consider a matrix $\mat A\in\CC$, and a preconditioner $\mat P\in\set{H}_m$. Then, there exists an initial vector $\mat s^{(0)}\in\R$ such that
\begin{equation}\label{eq:20210326}
\abs{\mat A-\mat P}\mat e\matle \mathfrak{M}(\mat P)\mat s^{(0)}.
\end{equation}
Furthermore, if the iteration~\eqref{eq:sk} is initiated by a vector $\mat s^{(0)}$ (for $k=0$) that satisfies~\eqref{eq:20210326}, then it holds the bound $\mu(\mat A,\mat P)\le\norm{\mat s^{(k)}}$ for all $k\ge 0$, and $\lim_{k\to\infty}\norm{\mat s^{(k)}}=\mu(\mat A,\mat P)$.
\end{proposition}

\begin{proof}
The existence of a vector $\mat s^{(0)}$ that satisfies~\eqref{eq:20210326} is immediately established upon setting $\mat s^{(0)}:=\mat s(\mat A,\mat P)$, cf.~\eqref{eq:s}. Now consider any vector $\mat s^{(0)}\in\R$ that fulfills~\eqref{eq:20210326}. Then, from~\eqref{eq:sk} with $k=0$, we have
\[
\abs{\diag(\mat P)}\left(\mat s^{(1)}-\mat s^{(0)}\right)
=-\mathfrak{M}(\mat P)\mat s^{(0)}+\abs{\mat A-\mat P}\mat e\matle\mat 0,
\]
which shows that $\mat s^{(1)}-\mat s^{(0)}\matle\mat 0$. Hence, by induction, since $\mathfrak{M}(\mat P)^{-1}\matge\mat 0$, cf.~fact (F3), we note that
\[
\mat s^{(k+1)}-\mat s^{(k)}=\abs{\diag(\mat P)^{-1}\off(\mat P)}\left(\mat s^{(k)}-\mat s^{(k-1)}\right)\matle\mat 0\qquad\forall~ k\ge 1.
\]
Using that $\sr{\abs{\diag(\mat P)^{-1}\off(\mat P)}}<1$, cf.~fact (F4), we infer that the iteration~\eqref{eq:sk} converges to $\mat s(\mat A,\mat P)$ from~\eqref{eq:s}. Moreover, from~\eqref{eq:s} and~\eqref{eq:sk} we deduce the identity
\begin{align*}
\mathfrak{M}(\mat P)\mat s(\mat A,\mat P)
&=\abs{\mat A-\mat P}\mat e
=\abs{\diag(\mat P)}\mat s^{(k+1)}-\abs{\off(\mat P)}\mat s^{(k)}\\
&=\mathfrak{M}(\mat P)\mat s^{(k+1)}+\abs{\off(\mat P)}\left(\mat s^{(k+1)}-\mat s^{(k)}\right),
\end{align*}
for all $k\ge 0$. Exploiting again that $\mathfrak{M}(\mat P)^{-1}\matge\mat 0$, we arrive at
\begin{align*}
\mat s(\mat A,\mat P)
&=\mat s^{(k+1)}+\mathfrak{M}(\mat P)^{-1}\abs{\off(\mat P)}\left(\mat s^{(k+1)}-\mat s^{(k)}\right)\matle\mat s^{(k+1)}.
\end{align*}
Since $\mat s(\mat A,\mat P)$ and $\mat s^{(k+1)}$ are both non-negative, the asserted bound follows.
\end{proof}

The ensuing result, which immediately follows from the previous one, allows for an estimate of the Sassenfeld index without solving the system~\eqref{eq:s}.

\begin{corollary}
Given a matrix $\mat A\in\CC$. Furthermore, let $\mat P\in\set{H}_m$ and $\mat v\in\RR$ such that 
\begin{equation}\label{eq:delta}
|\mat A-\mat P|\mat e\matle\mathfrak{M}(\mat P)\mat v.
\end{equation}
Then, it holds that $\mu(\mat A,\mat P)\le\norm{\mat v}$. 
\end{corollary}

\begin{example}[Jacobi preconditioner]
If $\mat P=\diag(\mat A)$ is non-singular then the bound~\eqref{eq:delta} is fulfilled for any vector $\mat v$ with components 
\[
v_i\ge\frac{1}{|a_{ii}|}\sum_{j\neq i}|a_{ij}|,\qquad 1\le i\le m,
\] 
and we have $\mu(\mat A,\diag(\mat A))\le \max_{1\le i\le m}v_i$. For instance, for the classical finite difference matrix
\begin{equation}\label{eq:FDM}
\mat A=\begin{pmatrix}
2 & -1& & \\[-1ex]
-1 & 2 & \ddots&\\
& \ddots &\ddots&-1\\
&&-1 & 2
\end{pmatrix}\in\RR,
\end{equation}
it is easily verified, for $m\ge 3$, that $\mu(\mat A,2\I)=1$.
\end{example}

\begin{example}[Gau\ss-Seidel preconditioner]
If $\mat P=\tril(\mat A)+\diag(\mat A)$ is non-singular then \eqref{eq:delta} translates into
\[
\sum_{j>i}|a_{ij}|\le |a_{ii}|v_i-\sum_{j<i}|a_{ij}|v_j,\qquad 1\le i\le m,
\]
which is essentially the recursive relation~\eqref{eq:SF} for Sassenfeld's original criterion. For the matrix $\mat A$ from~\eqref{eq:FDM}, it holds that $\mu(\mat A,\diag(\mat A)+\tril(\mat A))=1-2^{1-m}<1$. To give an example, for $m=10$, the iterative scheme~\eqref{eq:sk} with the initial vector $\mat s^{(0)}=\mat e$, cf.~\eqref{eq:e}, reaches the exact value of the Sassenfeld index (up to machine precision) after 9 iterations.
\end{example}

\section{Sassenfeld matrices}\label{sc:sm}

\subsection{Definition of Sassenfeld matrices}

We are now ready to introduce the notion of Sassenfeld matrices. Our definition, see Def.~\ref{def:SF} below, is motivated by the work~\cite{BaumannWihler:17}, where the special case of all matrices $\mat A\in\CC$ with $\mu(\mat A,\mat P)<1$, with $\mat P=\tril(\mat A)+\diag(\mat A)$ being the Gau\ss-Seidel preconditioner, has been discussed. In this specific situation, the system \eqref{eq:s} takes the (lower-triangular) form
\[
\abs{\diag(\mat A)}\mat s=\abs{\tril(\mat A)}\mat s+\abs{\triu(\mat A)}\mat e,
\]
which is a simple forward solve for~$\mat s$. Convergence of the Gau\ss-Seidel method is guaranteed if $\norm{\mat s}<1$; this is the key observation in Sassenfeld's original work~\cite{Sassenfeld:51}. 

More generally, for preconditioners $\mat P\in\set{H}_m$ in the current paper, we propose the following definition.

\begin{definition}[Sassenfeld matrices]\label{def:SF}
A matrix $\mat A\in\CC$ is called a \emph{Sassenfeld matrix} if there exists a preconditioner $\mat P\in\set{H}_m$ such that $\mu(\mat A,\mat P)<1$. The set of all Sassenfeld matrices in $\CC$ will be denoted by $\set{S}_m$.
\end{definition}

\begin{remark}
From Cor.~\ref{cor:inv}, for $\tau=0$, we immediately deduce that every Sassenfeld matrix is non-singular.
\end{remark}

The following proposition provides a condition number estimate for the preconditioned matrix $\imat P\mat A$ in terms of the Sassenfeld index. 

\begin{proposition}[Condition number bound]
Suppose that $\mat A\in\set{S}_m$ is a Sassenfeld matrix, and $\mat P\in\set{H}_m$ with $\mu(\mat A,\mat P)<1$. Then, for the condition number (with respect to the $\infty$-norm) the bound
\[
\kappa_\infty(\imat P\mat A)\le\frac{1+\mu(\mat A,\mat P)}{1-\mu(\mat A,\mat P)}
\] 
holds true.
\end{proposition}

\begin{proof}
Let $\mat C:=\imat P\mat A$. From Prop.~\ref{pr:bound}, we deduce the bound
\[
\norm{\mat C}
\le 1+\norm{\I-\imat P\mat A}\le 1+\mu(\mat A,\mat P).
\]
Moreover, applying a Neumann series, we deduce the estimate
\[
\norm{\mat C^{-1}}
=\norm{\left(\I-(\I-\mat C)\right)^{-1}}
\le\frac{1}{1-\norm{\I-\mat C}}
\le\frac{1}{1-\mu(\mat A,\mat P)}.
\]
This concludes the proof.
\end{proof}

\subsection{Characterization of Sassenfeld matrices ($\set{S}_m=\set{H}_m$)}

We will now establish the main result of this paper, which shows that a Sassenfeld matrices belongs to~$\set{H}_m$ and vice versa.

\begin{theorem}\label{thm:main}
For any $m\ge 1$ it holds $\set{H}_m=\set{S}_m$.
\end{theorem}

\begin{proof}
If $\mat A\in\set{H}_m$ then $\mu(\mat A,\mat A)=0$, i.e. $\mat A$ is a Sassenfeld matrix. Conversely, suppose that $\mat A\in\set{S}_m$, and select a preconditioner $\mat P\in\set{H}_m$ with $\mu(\mat A,\mat P)<1$. Then, 
writing~\eqref{eq:s} component-wise, there are non-negative real numbers $0\le s_i<1$, $i=1,\ldots, m$, such that
\[
|p_{ii}|s_i-\sum_{j\neq i}|p_{ij}|s_j=\sum_{j=1}^m|a_{ij}-p_{ij}|,\qquad i=1,\ldots,m.
\]
Letting
\begin{equation}\label{eq:di}
\delta_i:=\frac{1}{|a_{ii}|}\sum_{j=1}^m(1-s_j)|a_{ij}-p_{ij}|\ge0,\qquad i=1,\ldots,m,
\end{equation}
and rearranging terms, we observe the identity
\[
\delta_i|a_{ii}|+\sum_{j\neq i} \left(|p_{ij}|+|a_{ij}-p_{ij}|\right)s_j
=\left(|p_{ii}|-|a_{ii}-p_{ii}|\right)s_i,
\]
for each $i=1,\ldots, m$. Applying the triangle inequality on either side, it follows that
\begin{equation}\label{eq:auxs}
\sum_{j\neq i} |a_{ij}|s_j
\le |a_{ii}|(s_i-\delta_i),\qquad i=1,\ldots,m.
\end{equation}
Furthermore, recalling fact~(F2), there are positive numbers $u_1,\ldots,u_m>0$ such that
\begin{equation}\label{eq:auxz}
\sum_{j\neq i} |p_{ij}|u_j< |p_{ii}|u_i,
\end{equation}
for each $i=1,\ldots,m$.  Introduce positive numbers
$\t_i:=\alpha s_i+u_i>0$, $i=1,\ldots,m$, where $\alpha\ge 0$ will be specified later, see~\eqref{eq:alpha} below. Then, for $1\le i\le m$, we have
\begin{align*}
\sum_{j\neq i} |a_{ij}|\t_j
&=\alpha\sum_{j\neq i}|a_{ij}|s_j
+\sum_{j\neq i} |p_{ij}|u_j
+\sum_{j\neq i} \left(|a_{ij}|-|p_{ij}|\right)u_j.
\end{align*}
Employing~\eqref{eq:auxs} and \eqref{eq:auxz}, for each $i=1,\ldots,m$, we derive the estimate
\[
\sum_{j\neq i} |a_{ij}|\t_j
< \alpha(s_i-\delta_i)|a_{ii}|+|p_{ii}|u_i
+\sum_{j\neq i} \left(|a_{ij}|-|p_{ij}|\right)u_j.
\]
Thus, we obtain
\begin{equation}\label{eq:auxsz}
\sum_{j\neq i} |a_{ij}|\t_j<|a_{ii}|\t_i-\alpha\delta_i|a_{ii}|+u_i\left(|p_{ii}|-|a_{ii}|\right)
+\sum_{j\neq i} \left(|a_{ij}|-|p_{ij}|\right)u_j,
\end{equation}
for each $i=1,\ldots,m$. 
Now choose $\alpha\ge 0$ sufficiently large so that
\begin{equation}\label{eq:alpha}
\alpha\delta_i|a_{ii}|\ge u_i\left(|p_{ii}|-|a_{ii}|\right)
+\sum_{j\neq i} \left(|a_{ij}|-|p_{ij}|\right)u_j\qquad\forall i\in\mathcal{I},
\end{equation}
where $\set I$ signifies the set of all indices $1\le i\le m$ for which $\delta_i>0$ in \eqref{eq:di}; we let $\alpha=0$ if $\set I=\emptyset$.
In order to proceed, we distinguish two separate cases:
\begin{enumerate}[(i)]
\item If $\delta_i=0$ then exploiting that $0\le s_j<1$ for each $j=1,\ldots,m$, we notice from~\eqref{eq:di} that $a_{ij}=p_{ij}$ for all $j=1,\ldots,m$. Then, from~\eqref{eq:auxsz}, we infer that
$
\sum_{j\neq i} |a_{ij}|\t_j<|a_{ii}|\t_i,
$ for all $i\not\in\set{I}$.
\item Otherwise, if $\delta_i>0$ then recalling $\alpha$ from~\eqref{eq:alpha}, we obtain that 
$
\sum_{j\neq i} |a_{ij}|\t_j<|a_{ii}|\t_i$, for all $i\in\set{I}$.
\end{enumerate}
In summary, we conclude that
$
\sum_{j\neq i} |a_{ij}|\t_j<|a_{ii}|\t_i,
$
for each $i=1,\ldots,m$, which implies that $\mat A\in\set{H}_m$, cf.~fact~(F2).
\end{proof}

\subsection{Application to splitting methods}

In the context of linear solvers, the following generalization of Sassenfeld's result \cite{Sassenfeld:51} on the Gau\ss-Seidel scheme is an immediate consequence of Prop.~\ref{pr:bound}.

\begin{proposition}[Iterative solvers]
For a Sassenfeld matrix $\mat A\in\set{S}_m$ (or equivalently, for $\mat A\in\set{H}_m$, cf.~Thm.~\ref{thm:main}), and any given vector~$\mat b\in\mathbb{C}^m$, consider the linear system~\eqref{eq:Ax=b}. Then, for a preconditioner $\mat P\in\set{H}_m$ with $\mu(\mat A,\mat P)<1$, the iteration~\eqref{eq:it} converges to the unique solution of~\eqref{eq:Ax=b} for any starting vector $\mat x^{(0)}\in\C$. Furthermore, it holds the \emph{a priori} bound
\[
\norm{\mat x-\mat x_{n}}\le\mu(\mat A,\mat P)^n\norm{\mat x-\mat x_{0}}, 
\]
for any $n\ge 0$.
\end{proposition}

\begin{examples}
The above result recovers the well-known fact (see, e.g., \cite{Varga:00}) that, for H-matrices, the splitting method~\eqref{eq:it} converges for both the Jacobi as well as for the Gau\ss-Seidel preconditioners, i.e. for $\mat P=\diag(\mat A)$ and $\mat P=\diag(\mat A)+\tril(\mat A)$, respectively. In the latter case, this follows from Sassenfeld's paper~\cite{Sassenfeld:51}, see also~\cite{BaumannWihler:17}. Moreover, the former case is a consequence of fact~(F4).
\end{examples}

\section{Conclusions}\label{sc:concl}

Inspired by Sassenfeld's historical convergence criterion for the classical Gau\ss-Seidel scheme, we have introduced the notion of the \emph{Sassenfeld index} (with respect to H-matrix preconditioners), which, in turn, gives rise to the set of \emph{Sassenfeld matrices} discussed in this work. Our main result shows that Sassenfeld matrices are equivalent to the class of (non-singular) H-matrices, thereby yielding a new characterization for such matrices. Moreover, an iterative procedure for the computational verification of the proposed \emph{generalized Sassenfeld criterion} is provided.

\bibliographystyle{amsalpha}
\bibliography{myrefs}

\end{document}